\documentclass[10pt,reqno]{amsart}

\usepackage[a4paper,left=35mm,right=35mm,top=30mm,bottom=30mm,marginpar=25mm]{geometry}

\usepackage{amsmath}
\usepackage{amssymb}
\usepackage{amsthm}
\usepackage[latin1]{inputenc}
\usepackage{eurosym}
\usepackage[dvips]{graphics}
\usepackage{graphicx}
\usepackage{epsfig}
\usepackage{hyperref}
\usepackage{dsfont}
\usepackage[nocompress, space]{cite}

\usepackage[displaymath,mathlines]{lineno}

\allowdisplaybreaks

\usepackage{listings}

\usepackage{float}
\usepackage{media9}
\usepackage{movie15}
\usepackage{hyperref}

\usepackage{ifthen}

\makeindex

\renewcommand{\div}{\operatorname{div}}

\newcommand{\Rr}{{\mathbb{R}}}

\newcommand{\Nn}{{\mathbb{N}}}

\newcommand{\Tt}{{\mathbb{T}}}

\newcommand{\Ff}{{\mathcal{F}}}

\newcommand{\epsi}{\varepsilon}
\def\d{{\rm d}}
\def\dx{{\rm d}x}

\def\dt{{\rm d}t}
\def\leq{\leqslant}
\def\geq{\geqslant}

\numberwithin{equation}{section}

\newtheoremstyle{thmlemcorr}{10pt}{10pt}{\itshape}{}{\bfseries}{.}{10pt}{{\thmname{#1}\thmnumber{
#2}\thmnote{ (#3)}}}
\newtheoremstyle{thmlemcorr*}{10pt}{10pt}{\itshape}{}{\bfseries}{.}\newline{{\thmname{#1}\thmnumber{
#2}\thmnote{ (#3)}}}
\newtheoremstyle{defi}{10pt}{10pt}{\itshape}{}{\bfseries}{.}{10pt}{{\thmname{#1}\thmnumber{
#2}\thmnote{ (#3)}}}
\newtheoremstyle{remexample}{10pt}{10pt}{}{}{\bfseries}{.}{10pt}{{\thmname{#1}\thmnumber{
#2}\thmnote{ (#3)}}}
\newtheoremstyle{ass}{10pt}{10pt}{}{}{\bfseries}{.}{10pt}{{\thmname{#1}\thmnumber{
A#2}\thmnote{ (#3)}}}

\theoremstyle{thmlemcorr}
\newtheorem{theorem}{Theorem}
\numberwithin{theorem}{section}

\theoremstyle{thmlemcorr*}
\newtheorem{theorem*}{Theorem}
\newtheorem{lemma*}[theorem]{Lemma}
\newtheorem{corollary*}[theorem]{Corollary}
\newtheorem{proposition*}[theorem]{Proposition}
\newtheorem{problem*}[theorem]{Problem}
\newtheorem{conjecture*}[theorem]{Conjecture}

\theoremstyle{defi}
\newtheorem{definition}[theorem]{Definition}
\newtheorem{hyp}{Assumption}

\newtheorem{problem}{Problem}

\theoremstyle{remexample}
\newtheorem{remark}[theorem]{Remark}
\newtheorem{example}[theorem]{Example}

\newtheorem{teo}[theorem]{Theorem}
\newtheorem{lem}[theorem]{Lemma}
\newtheorem{pro}[theorem]{Proposition}

\theoremstyle{ass}

\begin{document}

\title[The Planning Problem With Potential]{Uniform Estimates For The Planning Problem \\With Potential}

\author{Tigran Bakaryan}
\address[T. Bakaryan]{
        King Abdullah University of Science and Technology (KAUST),
CEMSE Division, Thuwal 23955-6900, Saudi Arabia.}
\email{tigran.bakaryan@kaust.edu.sa}
\author{Rita Ferreira}
\address[R. Ferreira]{
        King Abdullah University of Science and Technology (KAUST),
CEMSE Division, Thuwal 23955-6900, Saudi Arabia.}
\email{rita.ferreira@kaust.edu.sa}
\author{Diogo Gomes}
\address[D. Gomes]{
        King Abdullah University of Science and Technology (KAUST),
CEMSE Division, Thuwal 23955-6900, Saudi Arabia.}
\email{diogo.gomes@kaust.edu.sa}


\thanks{T. Bakaryan, R. Ferreira, and D. Gomes were partially supported
by baseline and start-up funds from King Abdullah University
of Science and Technology (KAUST) OSR-CRG2017-3452. 
}

\begin{abstract}
        In this paper, we  study  a priori estimates
for a first-order
mean-field planning problem  with a potential.  In the theory
of mean-field games (MFGs), a priori estimates play a crucial
role to prove the existence of classical solutions. In particular,
uniform bounds for the density of
players' distribution and its inverse are of utmost
importance.  Here, we investigate a priori  bounds for those quantities for
a planning problem with a non-vanishing potential. The presence
of a potential
raises non-trivial difficulties, which
we overcome  by  exploring a
        displacement-convexity property for the  mean-field planning
problem with a potential together with Moser's
 iteration
method. We show that if the potential satisfies a certain smallness
condition, then a displacement-convexity property holds. This
property enables $L^q$ bounds for the density. In
the
one-dimensional case,  the displacement-convexity property also
gives $L^q$ bounds for the inverse of the density.   Finally,
using  these $L^q$ estimates  and  Moser's iteration
method,  we obtain   $L^\infty$ estimates for the density
of
the distribution of the players and its inverse. \end{abstract}

\maketitle

\section{Introduction}
The theory of mean-field games (MFGs) was proposed by J.-M. Lasry and P.-L. Lions (see \cite{ll1,ll2,ll3}) and, independently, by M.
Huang, R. Malham{\'e}, and P. Caines (see \cite{huang2006large}). These games  describe the interaction between identical rational agents, where 
each agent minimizes the same value function. A standard MFG is determined by a system of PDEs, a Hamilton--Jacobi and a Fokker--Planck equation:
\begin{equation}
\label{MFGs}
\begin{cases} 
-u_t-\varepsilon \Delta u+H(x,Du)=f(x,m) &  \\ m_t-\varepsilon \Delta m-\div (mH_p(x,Du))=0 & 
\end{cases} \text{in}\,\ (0,T)\times \Tt^d,  
\end{equation}
with the initial and terminal conditions 
\begin{equation}
\label{boundaryMFGs}
\begin{cases} 
m(0,x)=m_0(x) &  \\ u(T,x)=u_T(x)& 
\end{cases} \text{in}\,\ \Tt^d.
\end{equation}
Here, $u$ represents the value function of a typical agent and $m$ the distribution of the agents. Under mild condition on the problem data, the existence of weak solutions of \eqref{MFGs}--\eqref{boundaryMFGs} is addressed in \cite{cgbt} and in \cite{FeGoTa20} using monotonicity methods. 
 Regarding classical solutions, it is proved in \cite{Gomes2015b,GPM2,GPM3,Gomes2016c} that  \eqref{MFGs}--\eqref{boundaryMFGs} has a unique classical solution under suitable conditions on the problem data. 
A priori estimates play a crucial role  in the proof of the existence of classical solutions. In particular, the uniform boundedness
of the functions $m$ and $m^{-1}$, 
\begin{equation}\label{bounded_m_m-1}
\Vert m\Vert_{L^\infty}<\infty\enspace \text{ and }\enspace \Vert m^{-1}\Vert_{L^\infty}<\infty,
\end{equation}
is crucial to obtain classical solutions.

 In his lectures in  Coll\`ege de France  \cite{cursolionsplanning},  P.-L. Lions introduced {\em  mean-field planning problems}. This problem amounts to solving \eqref{MFGs}  with initial and terminal conditions only on the density, $m$; that is, 
 \begin{equation}
\label{boundaryP}
\begin{cases} 
m(0,x)=m_0(x) &  \\ m(T,x)=m_T(x)& 
\end{cases} \text{in}\,\ \Tt^d.
 \end{equation}
In those lectures,  P.-L. Lions proved the existence and uniqueness of classical solutions for the planning problem with a quadratic Hamiltonian, $H(x, p)=\frac{|p|^2}{2}$,   within both the second-order case ($\varepsilon >0$ in \eqref{MFGs})  and  the first-order
case ($\varepsilon =0$ in \eqref{MFGs}), and with $f = f(m)$ an increasing function  (see \cite{cursolionsplanning}). In \cite{porretta,porretta2}, A. Porretta proved the existence and uniqueness of weak solutions for the second-order case with a more general Hamiltonian. For the first-order case with $f = f(m)$ an increasing function, D. Gomes and T. Seneci explored in \cite{gomes2018displacement} the displacement convexity property to obtain $L^p$ and $L^\infty$ estimates.
Recently, the existence and uniqueness of weak solutions for the first-order case with a wide range of Hamiltonian has been addressed in \cite{OrPoSa2018,Tono2019}. 

We consider the case where the coupling function $f=f(x,m)$ is separated: $f(x,m)=g(m)-V(x)$. The potential, $V$, describes the spatial preferences of each agent.  In our setting, the potential
can also depend  on time, $V=V(t,x)$. More precisely, we investigate the following first-order mean-field planning problem with a time-dependent potential and
a quadratic Hamiltonian.
\begin{problem}\label{planningP} Suppose that $V\in C^\infty([0,T]\times\Tt^d)$ and that $g\in C^\infty(\Rr_0^+)$ is a non--decreasing function. Let $m_0$, $m_T\in C^\infty(\Tt^d)$ be  probability densities. Find  $(u,m)\in C^\infty([0,T]\times \Tt^d)\times C^\infty([0,T]\times \Tt^d)$ satisfying $m\geq 0$ and  
\begin{equation}
\label{planningPs}
\begin{cases} 
-u_t+\frac{|Du|^2}{2}+V(t,x)=g(m) & \text{in } (0,T)\times \Tt^d   \\ m_t-\div (mDu )=0  &\text{in } (0,T)\times \Tt^d  \\m(0,x)=m_0(x),\enspace m(T,x)=m_T(x) &\text{in }\Tt^d .
\end{cases}
\end{equation}
\end{problem}

In \cite{LavSantambrogio2017}, the authors use a flow interchange technique
to obtain $L^\infty$ estimates on the density of the solution of mean-field games without a potential.
This flow interchange technique is a discrete analog of the displacement 
convexity. In that same reference, a key technical tool is the Moser method
to iterate $L^p$ estimates and obtain $L^\infty$ estimates. This method is also used
here, although in a somewhat different manner. In particular,  
our focus is  on proving a priori bounds of the type \eqref{bounded_m_m-1}.
Such bounds were established in \cite{gomes2018displacement}  for solutions of the first-order mean-field planning problem without a potential, Problem \ref{planningP} with $V\equiv0$. In this manuscript, we concentrate on exploring similar a priori estimates for the first-order mean-field planning problem with a potential ($V\not\equiv0$). The presence of a potential raises technical difficulties in establishing a priori estimates, which we  overcome through new techniques that combine displacement convexity with Moser's iteration
method. 

Next, we state  our main results. We first outline our assumptions on the data of the
Problem~\ref{planningP}.
The first one is a smallness condition on the potential, $V$.
As we show at the end of   
 Section~\ref{sec3}, we cannot, in general, expect
 bounds  of the type \eqref{bounded_m_m-1} to hold without this smallness
condition.
\begin{hyp}\label{assumtionOf_V} There exists $p>0$ such that
the potential, $V:[0,T]\times \Tt^d \to \Rr$, satisfies 
        \begin{equation*}
        \Vert\Delta V\Vert_{L^\infty([0,T]\times\Tt^d)}<\frac{2}{T^2}\frac{1}{p}.
        \end{equation*}
\end{hyp}
 Further, we impose a positive lower bound on the planning problem initial-terminal data, $m_0$ and
$m_T$. This lower bound, together with the smoothness of \(m_0\)
and \(m_T\), guarantees that any  power of \(m_0^{-1}\) and \(m_T^{-1}\)  is an integrable function on \(\Tt^d\).
As it will become clear within our proofs, the value of those
integrals are key in our estimates. \begin{hyp}\label{assumtionOnBounds} There exists a positive constant,
$k_0$, such that the initial-terminal functions, $m_0$ and
$m_T$, satisfy
        \begin{equation*}
         m_0(x),\,m_T(x)\geq k_0>0,\quad x\in \Tt^d.
        \end{equation*} 
\end{hyp}  
\begin{teo}\label{mIsLq}  Let $(u,m)$ solve Problem \ref{planningP},
and suppose that Assumption \ref{assumtionOf_V}
 holds for some $p>0$. Then, there exists a positive  constant, $C$, depending only on the problem data and on $p$, such that 
        \begin{equation}\label{mIsLqeq}
        \max\limits_{t\in[0,T]}\Vert m\Vert_{L^{p+1}(\Tt^d)}\leq C.
        \end{equation}

Moreover, in the \(d=1\) case,  if Assumptions \ref{assumtionOf_V}
and \ref{assumtionOnBounds}   hold for some $p\geq2$, then there exists a positive  constant, $C$,  depending only on the problem data and on $p$, such that 
        \begin{equation}\label{mIsLqe-q}
        \max\limits_{t\in[0,T]}\Vert m^{-1}\Vert_{L^{p-1}(\Tt)}\leq C.
        \end{equation}
\end{teo}

\begin{remark}[On the \(p=0\) case]\label{rmk:p=0}
We observe that the estimate in \eqref{mIsLqeq} holds for \(p=0\) and for
an arbitrary smooth potential, \(V\in C^\infty([0,T]\times\Tt^d)\). In fact, by the mass-conservation
property of the Fokker--Planck
equation together with the initial condition, we have \(\int_{\Tt^d}
m\,\dx =1\) for all \(t\in[0,T]\). \end{remark}

In Section \ref{sec2},
we explore a  displacement-convexity  property of Problem~\ref{planningP}. We refer the
reader to  \cite{gomes2018displacement}
(and the references therein)  for further  insights on the
concept of displacement convexity for MFGs. Relying on this property,
we prove Theorem~\ref{mIsLq}.
\smallskip

Next, we address the  particular case of  Problem~\ref{planningP} corresponding to the coupling \(g(m) = m^\alpha\) for some \(\alpha>0\).
 For such  coupling functions, which feature many MFGs models, 
 we improve the   estimates in  Theorem~\ref{mIsLq}, as stated
below.

\begin{teo}\label{mIsLq-q+q} Let $(u,m)$
solve Problem~\ref{planningP}
with  \(g(m) = m^\alpha\) for some \(\alpha>0\). Suppose that Assumption \ref{assumtionOf_V}
 holds for some $p>0$. Then, there exists a positive  constant, $C$, depending only on the problem data and on \(p\), such that 
        \begin{equation*}
        \max\limits_{t\in[0,T]}\Vert m\Vert_{L^{\infty}(\Tt^d)} \leq C.
        \end{equation*}
        
Moreover, in the \(d=1\) case,  if Assumptions \ref{assumtionOf_V}
and \ref{assumtionOnBounds}   hold for some \(p>0\) with $p\geq2$ and $p>\alpha+1$, then there exists a positive  constant, $C$,  depending only on the problem data and on $p$, such that 
         \begin{equation*}
        \max\limits_{t\in[0,T]}  \Vert m^{-1}\Vert_{L^{\infty}(\Tt)} \leq C.
        \end{equation*}
\end{teo}

  We prove Theorem~\ref{mIsLq-q+q} in Section~\ref{sec3}  by combining the arguments
in the proof of Theorem~\ref{mIsLq}
with the Moser iteration method. We further show in   
 Section~\ref{sec3} that we cannot expect 
Theorem~\ref{mIsLq-q+q} to 
hold for general potentials. In fact, we exhibit in Example~\ref{exam}
     an instance of \eqref{planningPs}  with an unbounded potential, $V$,
for which the   solution $(u,m)$ is such that $m\in C^\infty([0,T]\times\Tt)$  and $m$ attains the zero. Thus, in particular,
  the inverse density function, $m^{-1}$, is unbounded.

 \section{Displacement Convexity for the Planning Problem with a Potential}
\label{sec2}

Here, we explore displacement-convexity properties for Problem \ref{planningP}, which will enable us to prove Theorem \ref{mIsLq}.

Let $(u,m)$ solve Problem \ref{planningP}.  
As shown in \cite{gomes2018displacement}, for certain functions $U:\Rr_0^+\to\Rr$,  the map
\begin{equation}\label{Uconv}
t\mapsto \int_{\mathbb{T}^d}^{}U(m(t,x))\,\dx
\end{equation}
 is convex when $V\equiv0$. The convexity of the  map in \eqref{Uconv} implies that we can control %
\begin{equation*}
\begin{aligned}
\max_{0\leq t\leq T} \int_{\mathbb{T}^d}^{}U(m(t,x))\,\dx
\end{aligned}
\end{equation*}
in terms of its values at \(t=0\) and \(t=T\). In contrast with the case  without potential, that property is, in general,  false for the case with a potential   (see Example \ref{exam}). 

Next, using this  displacement convexity,  we explore   conditions on $V$ and  $U$ under which  the maximum of the map in \eqref{Uconv} on a given interval is controlled by its values at the endpoints of the interval. 

First, we set 
 \begin{equation}\label{p}
 P(z)=zU^\prime(z)-U(z).
 \end{equation}
Accordingly,
\begin{equation*}
\begin{split}
\frac{\d}{\dt}\int_{\mathbb{T}^d}^{}U(m)\,\dx&=\int_{\mathbb{T}^d}^{}U^\prime(m)m_t\,\dx=\int_{\mathbb{T}^d}^{}U^\prime(m)\div (mDu )\,\dx\\&= \int_{\mathbb{T}^d}^{}\big(U^\prime(m)m\Delta u +U^\prime(m)Dm\cdot Du\big) \,\dx\\&=\int_{\mathbb{T}^d}^{}\big(U^\prime(m)m\Delta u +D(U(m)) \cdot Du\big) \,\dx\\&=
\int_{\mathbb{T}^d}^{}\big(U^\prime(m)m\Delta u -U(m)\Delta u\big)\,\dx=\int_{\mathbb{T}^d}^{}P(m)\Delta u\,\dx.
\end{split}
\end{equation*}
Differentiating  the preceding  identity, we get 
\begin{equation}\label{eq1}
\begin{split}
&\frac{\d^2}{\dt^2}\int_{\mathbb{T}^d}^{}U(m)\,\dx=\int_{\mathbb{T}^d}^{}\big(P^\prime(m)m_t\Delta u+P(m)\Delta u_t\big)\,\dx.
\end{split}
\end{equation}
On the other hand, applying $\Delta$  to the first equation in \eqref{planningPs}, we obtain 
\begin{equation*}
\Delta u_t=|D^2u|^2+DuD\Delta u+\Delta V-\div (g^{\prime}(m)Dm ).
\end{equation*}
Using this equality and taking into account the second equation in \eqref{planningPs}, we deduce from   \eqref{eq1} that
\begin{equation*}
\begin{split}
\frac{\d^2}{\dt^2}\int_{\mathbb{T}^d}^{}U(m)\,\dx=\int_{\mathbb{T}^d}^{}\big[&P^\prime(m)(\Delta u)^2 m+P^\prime(m)\Delta u Du\cdot  Dm+P(m)|D^2u|^2\\&+P(m)D\Delta u\cdot  Du+P(m)\Delta V-P(m)\div (g^{\prime}(m)Dm )\big]\,\dx.
\end{split}
\end{equation*}

To estimate the right-hand side of the preceding equality, we  observe that   integrating by parts  yields the identity\begin{align*}
&\int_{\mathbb{T}^d}^{}P(m)D\Delta u\cdot Du\,\dx- \int_{\mathbb{T}^d}^{}P(m)\div (g^{\prime}(m)Dm )\,\dx\\&\quad=-\int_{\mathbb{T}^d}^{}\div(P(m)Du)\Delta u\,\dx +\int_{\Tt^d} P'(m)g'(m) |Dm|^2\,\dx \\&\quad=\int_{\mathbb{T}^d}^{}\big(-P^\prime(m)Dm\cdot Du\Delta u -P(m)(\Delta u)^2 + P'(m)g'(m) |Dm|^2 
\big)\,\dx.
\end{align*}
Moreover, by the  Cauchy-Schwarz inequality, we have 
\begin{equation}\label{eq:C-S}
|D^2u|^2\geq\sum_{i=1}^{d}u_{x_ix_i}^2\geq\frac{1}{d}\bigg(\sum_{i=1}^{d}u_{x_ix_i} \bigg) ^2=\frac{1}{d}(\Delta u)^2.
\end{equation}
Hence, if \(U\) is such that \(P(m) \geq 0\), we deduce that
\begin{equation}\label{dispeq} 
\begin{split}
\frac{\d^2}{\dt^2}\int_{\mathbb{T}^d}^{}U(m)\,\dx\geq\int_{\mathbb{T}^d}^{}\Big[&(\Delta u)^2\left( P^\prime(m)m-P(m)+\tfrac{1}{d}P(m)\right)\\&+P^\prime (m)g^\prime (m)|D m|^2+P(m)\Delta V\Big]\,\dx.
\end{split}
\end{equation}

When $U$ is a power function, 
\begin{equation}\label{U}
U(z)=z^s,\quad s\geq 1,
\end{equation}
 from \eqref{dispeq}  and \eqref{p}, we get
\begin{equation}\label{eq3}
\frac{\d^2}{\dt^2}\int_{\mathbb{T}^d}^{}U(m)\,\dx \geq -|s-1|\Vert\Delta V\Vert_{L^\infty(\Tt^d)} \int_{\mathbb{T}^d}^{}U(m)\,\dx.
\end{equation}
\begin{remark}\label{rems}
        In one-dimensional case, $d=1$, the estimate \eqref{eq:C-S} holds with equality. In particular, we do not need the condition \(P(m)\geq 0\) to get \eqref{dispeq} (which then holds with equality). Hence, a direct computation shows that, in \(d=1\) case, \eqref{eq3} holds  for \(U(z)=z^s\) with  $s \in  \Rr \setminus      (0,1)$.
\end{remark}

Next, we introduce a simplified notation to denote the class
of functions that
satisfy a condition of the  type  \eqref{eq3}. Under such type of condition, the subsequent lemma provides a smallness constraint on  \(V\) under which  the maximum of map in \eqref{Uconv} on a given
interval is controlled by its values at the endpoints of the interval,  as detailed in the proof of  Theorem~\ref{mIsLq}.
\begin{definition}\label{def} Given  $a$, $b$, $c\geq 0$,  we denote by  $\Ff_a^b(c)$  the set of all  functions  $f\in C^2([0,T])$ that are  non-negative and satisfy 
         \begin{equation}\label{eq4}
         \begin{cases} 
         f^{\prime\prime}(t)+cf(t)\geq 0\enspace \text{ for all }    t\in [0,T],   \\ f(0)=a,\enspace f(T)=b.& 
         \end{cases}
         \end{equation}
\end{definition}

\begin{lem}\label{UniBoundin} Suppose that $0<\varepsilon\leq2$ and $a$, $b\geq 0$. Let  $0\leq c\leq\frac{2-\varepsilon}{T^2}$. Then, the family of functions  $\Ff_a^b(c)$ introduced in Definition~\ref{def} is uniformly bounded; more precisely,  for any $f\in \Ff_a^b(c)$, we have
        \begin{equation*}
        0\leq f(t)\leq \frac{2(a+b)}{\varepsilon}\enspace \text{ for all } t\in[0,T].
        \end{equation*}
\end{lem}
\begin{proof}
        Set
         \begin{equation}\label{prop1fmax}
        M=\max\limits_{t\in[0,T]}f(t),
        \end{equation}
        and let
        \begin{equation*}
        h(t)=f(t)+kt^2, \enspace 
        \end{equation*} 
        where \(k=\frac{cM}{2}\). 
        We claim that $h$ is  convex. In fact, by  \eqref{eq4}  and \eqref{prop1fmax}, we have
        \begin{equation*}
        h^{\prime\prime}(t)=f^{\prime\prime}(t)+2k=f^{\prime\prime}(t)+cf(t)+2k-cf(t) =f^{\prime\prime}(t)+cf(t)+c(M - f(t)) \geq0 ,
        \end{equation*}
        which proves the claim. 

By the convexity and non-negativity of $h$, we conclude that, for all \(t\in[0,T]\),
we have        \begin{equation*}
        f(t)\leq h(t)\leq h(0)+h(T)=f(0)+f(T)+k T^2=a+b+\frac{cM}{2}T^2.
        \end{equation*}
        Taking the maximum over \(t\in[0,T]\) in the preceding estimate, we get         \begin{align*}
        &M\leq a+b+\frac{cM}{2}T^2, 
        \end{align*}
from which we deduce that \(M  \leq \tfrac{2(a+b)}{2-cT^2} \leq \tfrac{2(a+b)}{\epsi}\) because   $c\leq\frac{2-\varepsilon}{T^2}$. \end{proof}
\begin{remark}
        The claim in Lemma~\ref{UniBoundin} is false for an arbitrary positive  constant \(c\).  For instance, let $c\geq \frac{\pi^2}{T^2}$ and
\begin{equation*}
        f_k(t)=k \sin \tfrac{\pi t}{T}  +1,\quad k\in \Nn. 
        \end{equation*}
Then, \(f_k\in\Ff_1^1(c)\)
  for all \(k\in\Nn\), which shows that, in this case, the claim
in Lemma~\ref{UniBoundin}  fails for any fixed constant \(\epsi>0\) and $c\geq \frac{\pi^2}{T^2}$. 
\end{remark}
 
\begin{proof}[Proof of Theorem~\ref{mIsLq}] We start by proving
 the  estimate in \eqref{mIsLqeq}.  Let $s=p+1$. Then, from Assumption~\ref{assumtionOf_V}, it follows that there exists $0<\varepsilon< 2$ such that 
        \begin{equation*}
|s-1|\Vert\Delta V\Vert_{L^\infty([0,T]\times\Tt^d)}\leq\frac{2-\varepsilon}{T^2}.
        \end{equation*}
        Applying   Lemma \ref{UniBoundin}  to \(f(t):=\int_{\Tt^d} m^s(x,t)\,\dx\) and taking into account \eqref{U}--\eqref{eq3}, we deduce that 
         \begin{align}
         &\int_{\mathbb{T}^d}^{}m^s\,\dx\leq\frac{2}{\varepsilon}\left(          \int_{\mathbb{T}^d}^{}m_0^s\,\dx+\int_{\mathbb{T}^d}^{}m_T^s\,\dx\right),\label{eq:bddmainthm}
         \end{align}
         which together with the smothness of \(m_0\) and \(m_T\)
         concludes the proof of  \eqref{mIsLqeq}.

The proof of the estimate in \eqref{mIsLqe-q}
is analogous, as we outline next. In this case, we take \(s=-p+1\) and observe that
\(s\leq -1\) for \(p\geq 2\). The conclusion follows by applying Lemma~\ref{UniBoundin}  using Remark~\ref{rems}. Note also that Assumption~\ref{assumtionOnBounds}
guarantees that the right-hand side of \eqref{eq:bddmainthm} is finite whenever
\(s<0\).       
\end{proof}

\section{Further  Estimates }
\label{sec3}
In this section, we consider the mean-field planning problem in Problem~\ref{planningP}
with  \(g(m) = m^\alpha\) for some \(\alpha>0\). In this  case, we  establish $L^\infty$ estimates for  $m$ and $m^{-1}$, as stated in Theorem \ref{mIsLq-q+q}, which we  prove by combining the next two propositions. 

\begin{pro}\label{mIsLinfnity-q} Let $(u,m)$ solve Problem~\ref{planningP}
with  \(g(m) = m^\alpha\) for some \(\alpha>0\) and $d=1$. Suppose that Assumption \ref{assumtionOnBounds} holds and that there
exists  $r\geq 1$ with \(r>\alpha \) for which we can find  a positive constant, $c$, such that
        \begin{equation}\label{m_int_r}
        \max\limits_{t\in[0,T]} \int_{\Tt}^{}\frac{1}{m^{r}}\,\dx<c.
        \end{equation}
        Then, there exists a positive  constant, $C$, depending only on the problem data and on the constant in \eqref{m_int_r}, such that 
        \begin{equation}\label{mIsLinfnity-q-st}
        \max\limits_{t\in[0,T]} \Vert m^{-1}\Vert_{L^{\infty}(\Tt)}\leq C.
        \end{equation}
\end{pro}
\begin{proof} 
 To simplify the notation, throughout this proof, we  denote by the same letter \(C\) any positive
constant that depends only on  the
problem data  or on the constant in \eqref{m_int_r} or  on universal constants such as the constant
in the
Sobolev inequality. However, we   keep track of the relevant
power dependencies of these constants.  
Moreover, we may assume without
loss of generality,  that any such constant \(C\)
satisfies \(C\geq1\).
 
For  $s\geq 1$, set 
        \begin{equation}\label{mIsLinfnity-q_assum}
        M_s=\max\limits_{t\in[0,T]} \int_{\Tt}^{}\frac{1}{m^{s}}\,\dx.
        \end{equation}
Fix 
\begin{equation}\label{condFor_q}
q>r+\alpha,
\end{equation}
and let $\ell=\frac{2r}{q-\alpha}$.
By  \eqref{m_int_r},
we have \(M_r<\infty\) and, without
loss of generality, we may assume that
\(M_r\geq 1\).

As we are in one-dimensional case, $d=1$, by Remark \ref{rems} and \eqref{dispeq} with \(s=-q\), we have
        \begin{equation}\label{mIsLinfnity-q_eq1}
\begin{split}
\frac{\d^2}{\dt^2}\int_{\Tt}^{}\frac{1}{m^q}\,\dx &= q(q+1)\int_{\mathbb{T}}^{}\frac{(\Delta
u)^2}{m^{q}}\,\dx+\alpha q(q+1) \int_{\Tt}^{}\frac{|Dm|^2}{m^{q+2-\alpha}}\,\dx-(q+1) \int_{\Tt}^{}\frac{\Delta V}{m^{q}}\,\dx\\
&\geq \alpha q(q+1) \int_{\Tt}^{} \frac{|Dm|^2}{m^{q-\alpha+2}} \,\dx-(q+1)C \int_{\Tt}^{}\frac{1}{m^{q}}\,\dx\\
&=4\alpha \frac{q(q+1)}{(q-\alpha)^2} \int_{\Tt}^{} \left| D\left( \frac{1}{m^{\frac{q-\alpha}{2}}}\right) \right| ^2 \dx-(q+1)C \int_{\Tt}^{}\frac{1}{m^{q}}\,\dx.
\end{split}
        \end{equation}
        
Set 
        \begin{equation*}
f(t)=\int_{\Tt}^{} \left|  D\left( \frac{1}{m^{\frac{q-\alpha}{2}}}\right) \right| ^2 \dx.
        \end{equation*}
        
According to the generalized Poincar{\'e} inequality (see 
        \cite[Proposition~4.10]{GPV}), for any $0<a<2$, there exists a positive constant, $C_a\geq 1$, such that, for any function $h\in W^{1,2}(\Tt)$, we have 
        \begin{equation}\label{PoinGen}
\left( \int_{\Tt}^{}|h|^2\,\dx\right)^\frac{1}{2} \leq C_a\left[ \left( \int_{\Tt}^{}|Dh|^2\,\dx\right)^\frac{1}{2}+\left( \int_{\Tt}^{}|h|^a\,\dx\right)^\frac{1}{a}\right]. 
        \end{equation}
 Taking $h=m^{\frac{-q+\alpha}{2}}$ and $a=\ell=\frac{2r}{q-\alpha}$ in \eqref{PoinGen}, we obtain
\begin{equation}\label{Poincare}
        \int_{\Tt}^{}\frac{1}{m^{q-\alpha}}\,\dx \leq C^2_{\ell}   \left(  f^{\frac{1}{2}}(t)+M_r^{\frac{q-\alpha}{2r}} \right)^2 \leq2C^2_\ell \left (f(t)
+ M_r^{\frac{q-\alpha}{r}}\right).
\end{equation}
On the other hand, by Morrey's embedding  theorem (see \cite[Section 5.6, Theorem 4]{E6}), we have
\begin{equation}
\label{Morrey}
\left\Vert \frac{1}{m^{\frac{q-\alpha}{2}}}\right\Vert_{L^{\infty}(\Tt)}\leq C\left( \int_{\mathbb{T}}^{}\frac{1}{m^{q-\alpha}}\,\dx +f(t) \right)^{\frac{1}{2}}.
\end{equation}
Raising the preceding estimate to the power of \(2/(q-\alpha)\)
first, and then using  \eqref{Poincare}, we deduce that
\begin{equation}\label{Morry+Poin}
\left\Vert \frac{1}{m}\right\Vert_{L^{\infty}(\Tt)}\leq C^{\frac{2}{q-\alpha}} C_\ell^{\frac{2}{q-\alpha}}\left( M_r^{\frac{1}{r}} +(f(t))^{\frac{1}{q-\alpha}} \right). 
\end{equation}

On the other hand, by \eqref{condFor_q}, we have
\begin{equation*}
\theta=\frac{r}{q}\in(0,1).
\end{equation*}
Hence,  H\"older's interpolation inequality yields
\begin{equation}\label{Holder1}
\left\Vert \frac{1}{m}\right\Vert_{L^{q}(\Tt)}\leq\left\Vert \frac{1}{m}\right\Vert^\theta _{L^{r}(\Tt)} \left\Vert \frac{1}{m}\right\Vert^{1-\theta} _{L^{\infty}(\Tt)} \leq M_r^{\frac{\theta}{r}}\left\Vert \frac{1}{m}\right\Vert^{1-\theta}
_{L^{\infty}(\Tt)}.
\end{equation}
Then, setting
\begin{equation}\label{gamma<0}
\gamma=\frac{q(1-\theta)}{q-\alpha}=\frac{q-r}{q-\alpha},
\end{equation}
 we conclude from \eqref{Morry+Poin} and \eqref{Holder1} that
\begin{equation}\label{eq:estLqd=1}
\int_{\Tt}^{}\frac{1}{m^{q}}\,\dx\leq C^{2\gamma+1} C_\ell^{2\gamma}M_r^{\frac{q\theta}{r}}\Big( M_r^{\frac{q(1-\theta)}{r}}+(f(t))^{\gamma}\Big).
\end{equation}
Observing that  \(\frac{q\theta}{r}=1\) and
\(\frac{q(1-\theta)}{r} = \frac{q-r }{r}\),  estimates   \eqref{eq:estLqd=1} and \eqref{mIsLinfnity-q_eq1} yield
\begin{equation}\label{mIsLinfnity-q_eq1'}
\frac{\d^2}{\dt^2} \int_{\Tt}^{}\frac{1}{m^{q}}\,\dx \geq 4\alpha \frac{q(q+1)}{(q-\alpha)^2}f(t) -(q+1)C^{2\gamma+2} C_\ell^{2\gamma}M_r\left( M_r^{\frac{q-r }{r}}+(f(t))^{\gamma}\right). 
\end{equation}

Next, we  estimate the right-hand side of \eqref{mIsLinfnity-q_eq1'} using Young's inequality with $\varepsilon$; namely, the estimate 
\begin{equation}\label{Young}
ab\leq \varepsilon a^\sigma+ C(\varepsilon) b^{\frac\sigma{\sigma-1}} 
\end{equation}
that is valid for all \(a\),
\(b\geq0\), \(\epsi>0\), and \(1<\sigma<\infty\), with $C(\varepsilon)=\frac{\sigma-1}{\sigma}(\varepsilon \sigma)^{-\frac{1}{\sigma-1}}$. Note that \(\gamma\in(0,1)\) because $r>\alpha$ and \(q>\min\{r,\alpha\}\)
by \eqref{condFor_q}.  Then, taking $\sigma=\gamma^{-1}$,   $a=(f(t))^{\gamma}$, $ b=1$, and  
\begin{equation}\label{epsilond>2}
\varepsilon=4\alpha \frac{q}{(q-\alpha)^2}\frac{1}{C^{2\gamma+2} C_\ell^{2\gamma}M_r}.
\end{equation}
in \eqref{Young}, we conclude that
\begin{equation}\label{Youngf}
(f(t))^{\gamma}\leq \varepsilon f(t) +C(\varepsilon),
\end{equation}
where
\begin{equation}\label{Cepsiond>2}
C(\varepsilon)=(1-\gamma)\left( \frac{\gamma}{\varepsilon}\right) ^{\frac{\gamma}{1-\gamma}}= (1-\gamma)\gamma^{\frac{\gamma}{1-\gamma}}\left( \frac{ (q-\alpha)^2{C}^{2\gamma+2}C_\ell^{2\gamma}M_r}{4\alpha q}\right)^{\frac{\gamma}{1-\gamma}}.
\end{equation}
Hence, using \eqref{Youngf} in \eqref{mIsLinfnity-q_eq1'} yields
\begin{equation}\label{mIsLinfnity-q_eq2_1}
\begin{split}
\frac{\d^2}{\dt^2}\int_{\Tt}^{}\frac{1}{m^q}\,\dx \geq -(q+1){C}^{2\gamma+2}C_\ell^{2\gamma} \left( M_r^{\frac{q }{r}}+M_rC(\varepsilon)\right).
\end{split}
\end{equation}
From the conditions \(0<\gamma<1\) and \(q>\alpha>0\), we get the estimates 
\begin{equation}\label{esti_cost1}
(1-\gamma)\gamma^{\frac{\gamma}{1-\gamma}}<1\quad\text{ and }\quad  \frac{(q-\alpha)^2}{4\alpha q}<\frac{q+1}{\alpha}.
\end{equation}
Because \(q+1\leq 2q\), it follows
from \eqref{esti_cost1} and \eqref{Cepsiond>2}  that
\begin{equation*}
\begin{aligned}
C(\epsi) \leq \left(\frac{2q}{\alpha} {C}^{2\gamma+2}C_\ell^{2\gamma}M_r \right)^{\frac{\gamma}{1-\gamma}} \leq  \left(q {C}^{2\gamma+3}C_\ell^{2\gamma}M_r\right)^{\frac{\gamma}{1-\gamma}},
\end{aligned}
\end{equation*}
which combined with \eqref{mIsLinfnity-q_eq2_1} yields
      \begin{equation}\label{mIsLinfnity-q_eq2_11}
\begin{split}
\frac{\d^2}{\dt^2}\int_{\Tt^d}^{}\frac{1}{m^q}\,\dx \geq - q{C}^{2\gamma+3}C_\ell^{2\gamma}M_r^{ \frac{q}{r}} -\left(q{C}^{2\gamma+3}C_\ell^{2\gamma}M_r\right)^\frac{1}{1-\gamma}.
\end{split}
\end{equation}

Further, taking into account that  \(\frac{1}{1-\gamma}=\frac{q-\alpha}{r-\alpha}>\frac{q}{r}>1\) and 
 \(q\), \(C\), \(C_\ell\), \(M_r\geq 1\), we deduce that
\begin{equation*}
\begin{split}
\frac{\d^2}{\dt^2}\int_{\Tt}^{}\frac{1}{m^q}\,\dx \geq -2\left( q{C}^{2\gamma+3}C_\ell^{2\gamma}\right)^{\frac{1}{1-\gamma}}M_r^{\frac{1}{1-\gamma}}.
\end{split}
\end{equation*}
Defining
\begin{equation*}
h(t)=\int_{\Tt}^{}\frac{1}{m^q}\,\dx +\left( q{C}^{2\gamma+3}C_\ell^{2\gamma}\right)^{\frac{1}{1-\gamma}}M_r^{\frac{1}{1-\gamma}}t^2,
\end{equation*}
the preceding estimate gives that \(h\) is a   (non-negative)
convex
function. Thus, \(h(t)\leq h(0) + h(T)\), which, together with
Assumption \ref{assumtionOnBounds}, implies that
\begin{equation*}
\begin{aligned}
\int_{\Tt}^{}\frac{1}{m^q}\,\dx &\leq \int_{\Tt}^{}\frac{1}{m_0^q}\,\dx+\int_{\Tt}^{}\frac{1}{m_T^q}\,\dx+\left( q{C}^{2\gamma+3}C_\ell^{2\gamma}\right)^{\frac{1}{1-\gamma}}M_r^{\frac{1}{1-\gamma}}T^2\\ &\leq 2\left( q{C}^{2\gamma+4}C_\ell^{2\gamma}\right)^{\frac{1}{1-\gamma}}M_r^{\frac{1}{1-\gamma}}.
\end{aligned}
\end{equation*}
Consequently,  recalling the notation introduced
in \eqref{mIsLinfnity-q_assum}, we have 
\begin{equation}\label{eq:m^-1}
M_q\leq \left( q{C}^{2\gamma+5}C_\ell^{2\gamma}\right)^{\frac{1}{1-\gamma}}M_r^{\frac{1}{1-\gamma}}.
\end{equation}
We observe that from \eqref{eq:m^-1}, the arguments above show
that for any \(q\) satisfying \eqref{condFor_q}, there exits a positive constant, \(C_{q,r}\),
depending only on \(r\), \(M_r\), 
\(q\), and  \(C\), such that
\begin{equation}
\label{eq:boundsbeta}
\begin{aligned}
M_q
\leq C_{q,r}.
\end{aligned}
\end{equation}

Set    
\begin{equation}\label{beta_cond}
\beta=\frac{3}{2},
\end{equation}
and, for each \(n\in\Nn\), define $q_n=\beta^n$. Note that $q_n\to\infty$ as \(n\to\infty \) because $\beta >1$. Thus, there exists $n_0\in\Nn$ such that for all $n\geq n_0$, we have 
\begin{equation}\label{eq:onqn}
 q_{n+1}>q_{n}+\alpha,  \quad  q_{n}> \alpha+r.
\end{equation} 
  In view of \eqref{eq:onqn}, we argue as before to conclude that
    \eqref{eq:m^-1} holds with $q=q_{n+1}$ and  $r=q_n$ for all $n\geq n_0$. Thus,
\begin{equation}\label{M_n_seq}
M_{q_{n+1}}\leq  \left( q_{n+1} {C}^{2\gamma_n+5}C_{\ell_n}^{2\gamma_n}\right)^\frac{1}{1-\gamma_n} M_{q_n}^{ \frac{1}{1-\gamma_n}},
\end{equation}
where
\begin{equation}\label{thetad>2_n}
\ell_n=\frac{2q_n}{q_{n+1}-\alpha}\enspace \text{ and }\enspace \gamma_n=\frac{q_{n+1}-q_n}{q_{n+1}-\alpha}.
\end{equation}

Recalling that  $q_n=\beta^n$ and \eqref{beta_cond}, we get
\begin{equation}\label{gamma_lim}
\lim\limits_{n\to\infty} \ell_n = \frac43\enspace \text{ and }\enspace\lim\limits_{n\to\infty}(1-\gamma_n)=\frac{2}{3}.
\end{equation}

From  \eqref{gamma_lim}, \eqref{thetad>2_n}, and \eqref{eq:onqn}, we can find  $N_0\geq n_0$+1, such that  
\begin{equation}\label{some_ineq}
\frac{4}{3}<\ell_n<\frac53 \enspace \text{ and
}\enspace\frac{1}{1-\gamma_n}<2
\end{equation}
for all $n\geq N_0$.
Relying on \eqref{some_ineq}, we may reduce  \eqref{M_n_seq} to
\begin{equation}\label{M_n_seq_red}
M_{q_{n+1}}\leq Cq_{n+1}^2 M_{q_{n}}^{ \frac{1}{1-\gamma_n}},
\quad n\geq N_0. \end{equation}

To complete the proof, it is enough to prove the boundedness of the sequence  $\{ M_{q_{n}}^{\frac{1}{q_{n}}}\}_{n\geq N_0} $.
Taking $q=q_{N_0}$ in \eqref{eq:boundsbeta}, we get
\begin{equation}\label{M_q1_bound}
M_{q_{N_0}}^{\frac{1}{q_{N_0}}}\leq C_{q_{N_0},r}.
\end{equation} 

For \(n\), \(k\in\Nn\), we define
\begin{equation}\label{def_psi}
\varPhi_k^n=\prod_{j=k}^{n}\frac{1}{1-\gamma_{j}}
\enspace \text{ and } \enspace
\varPsi_{n}=\sum_{k=N_0+1}^{n}\varPhi^n_{k},
\end{equation}
with the standard convention that an
empty product  equals 1 and an empty sum
equals 0. Note that
\begin{equation*}
\begin{aligned}
\varPhi_n^n =\frac{1}{1-\gamma_n}, \enspace \frac{1}{1-\gamma_{n+1}}\varPhi_k^n
= \varPhi_k^{n+1}, \enspace  \frac{1}{1-\gamma_{n+1}}(1+\varPsi_{n})=\varPhi_{n+1}^{n+1} +\sum_{k=N_0+1}^{n} \varPhi_k^{n+1} = \varPsi_{n+1}.
\end{aligned}
\end{equation*}
Using these identities, the  recurrence relation in \eqref{M_n_seq_red}, and a mathematical 
induction argument, we obtain
\begin{equation}\label{M_n_seq_M_1}
M_{q_{n+1}}^{\frac{1}{q_{n+1}}}\leq C^{\frac{1}{q_{n+1}}(1+\varPsi_{n})}\left( M_{q_{N_0}}^{ \varPhi^n_{N_0}}\right)^{\frac{1}{q_{n+1}}} \left( q_{n+1}^2 \prod_{k=N_0+1}^{n}q_{k}^{2\varPhi^n_{k}}\right)^{\frac{1}{q_{n+1}}} \end{equation}
for all \(n\geq N_0\).

Next, we  estimate the three multiplicative factors on the right-hand side of \eqref{M_n_seq_M_1} separately.
From \eqref{eq:onqn}, we know that $q_{j}-\alpha>q_{j-1}$ for all $j\geq N_0$.  Consequently, recalling \eqref{thetad>2_n} and the definition
\(q_n = \beta^n\), we have\begin{equation*}
\frac{1}{\beta(1-\gamma_{j})}=1+\frac{\beta\alpha-\alpha}{\beta(\beta^{j}-\alpha)}<1+\frac{\alpha}{\beta^{j}-\alpha}<1+\frac{\alpha}{\beta^{j-1}}.
\end{equation*}
 Thus, for all \(k\geq N_0\), we conclude
that \begin{equation*}
\begin{aligned}
\prod_{j=k}^{n}\frac{1}{\beta(1-\gamma_{j})}&<\prod_{j=k}^{n}\Big( 1+\frac{\alpha}{\beta^{j-1}}\Big) =
\prod_{j=0}^{n-k}\Big(
1+\frac{\alpha}{\beta^{k-1}}\frac1{\beta^j}\Big)
  \\&<\prod_{j=0}^{\infty}\Big( 1+\frac{\alpha}{\beta^{N_0}}\Big(\frac1{\beta}\Big)^j\Big)=\Big( -\frac{\alpha}{\beta^{N_0}};\beta^{-1}\Big),
\end{aligned}
 \end{equation*}
 where  \((a;\mathfrak{q})\) denotes the \(\mathfrak{q}\)-Pochhammer symbol (see, for example, \cite{gasper_rahman_2004}), which is a finite number for all \(a\in\Rr\)
and $\mathfrak{q}\in (0,1)$. So, $\left( -\alpha\beta^{-N_0};\beta^{-1}\right) =\rho<\infty$ because $\beta^{-1}<1$.  Hence, for all \(k\geq N_0\), we have the following estimate for  $\varPhi^n_k$:
 \begin{equation}\label{phi_es}
\varPhi^n_k=\beta^{n-k+1}\prod_{j=k}^{n}\frac{1}{\beta(1-\gamma_{j})}<\rho\beta^{n-k+1}.
 \end{equation}
 From \eqref{phi_es} and \eqref{def_psi}, we deduce that
  \begin{equation}\label{psi_es}
\varPsi_{n}=\sum_{k=N_0+1}^{n}\varPhi^n_{k}<\rho\sum_{k=N_0+1}^{n}\beta^{n-k+1} =\rho\,\frac{\beta(\beta^{n-N_0}-1)}{\beta-1}.
 \end{equation}
 
Consequently, recalling \eqref{beta_cond}, for all \(n\geq N_0\),
we have
\begin{equation*}
\begin{aligned}
\frac{\varPhi^n_{N_0}}{q_{n+1}}=\frac{\varPhi^n_{N_0}}{\beta^{n+1}}<\rho\beta^{-N_0}<\rho
\enspace \text{ and }\enspace \frac{\varPsi_{n}}{q_{n+1}}=\frac{\varPsi_{n}}{\beta^{n+1}}<\frac{\rho}{\beta-1}\beta^{-N_0}\big(1-\beta^{N_0-n})<2\rho,
\end{aligned}
\end{equation*}
from which we obtain the following estimates for the first two multiplicative factors on the right-hand
side of \eqref{M_n_seq_M_1}:  
 \begin{equation}\label{two_es}
C^{\frac{1}{q_{n+1}}(1+\varPsi_{n})}<C^{\frac{1}{q_{n+1}}+2\rho}
\enspace \text{ and
}\enspace \Big( M_{q_{N_0}}^{ \varPhi^n_{N_0}}\Big)^{\frac{1}{q_{n+1}}} <M_{q_{N_0}}^{ \rho}.
 \end{equation}
 
To estimate the third multiplicative factor on the right-hand side of \eqref{M_n_seq_M_1}, we observe that    \eqref{phi_es}, together with the condition \(\beta>1\),  
implies that
\begin{equation*}
\begin{aligned}
q_{n+1}^2 \prod_{k=N_0+1}^{n}q_{k}^{2\varPhi^n_{k}} = \beta^{2(n+1)}
\beta^{ \sum_{k=N_0+1}^{n} 2k\varPhi^n_{k}} &< \beta^{2(n+1)(\rho+1)}
\beta^{\sum_{k=N_0+1}^{n} 2k(\rho+1)\beta^{n-k+1}} \\ &= 
\beta^{2(\rho+1)\sum_{j=0}^{n-N_0}(n+1-j)\beta^j}.
\end{aligned}
\end{equation*}
Then, because 
 \begin{equation*}
 \begin{aligned}
\sum_{j=0}^{n-N_0}(n+1-j)\beta^j & =\frac{1}{(\beta-1)^2}\left(
\beta(N_0 + 1)\beta^{n-N_0+1}-N_0\beta^{n-N_0+1}-(n+1)(\beta-1) -\beta\right)\\
&< \frac{\beta(N_0 + 1)}{(\beta-1)^2}\beta^{n+1},
\end{aligned}
 \end{equation*}
 we conclude that
  \begin{equation}\label{pro_es_qn}
 \left( q_{n+1}^2 \prod_{k=N_0+1}^{n}q_{k}^{2\varPhi^n_{k}(\beta)}\right)^\frac{1}{\beta^{n+1}} <\beta^{2(\rho+1)\frac{\beta(N_0 + 1)}{(\beta-1)^2}}.
 \end{equation} 
 
Finally, from \eqref{M_q1_bound},   \eqref{M_n_seq_M_1},    \eqref{two_es}, and \eqref{pro_es_qn}, we conclude that   $\{ M_{q_{n}}^{\frac{1}{q_{n}}}\}_{n\geq N_0} $ is a bounded sequence.
 \end{proof}

The next theorem gives the uniform boundedness of the density function, $m$.
\begin{pro}\label{mIsLinfnity+q} Let $(u,m)$ solve Problem~\ref{planningP}
with  \(g(m) = m^\alpha\) for some \(\alpha>0\). Then, 
there exists a positive constant, $C$,  depending only on the problem data, such that
        \begin{equation*}
        \max\limits_{t\in[0,T]} \Vert m\Vert_{L^{\infty}(\Tt^d)}\leq C.
        \end{equation*}
\end{pro}
\begin{proof} Proposition~\ref{mIsLinfnity+q} can be proved 
with similar arguments to those in the previous proof; thus, 
we only highlight the main differences.

First, we observe that, by the mass-conservation property (see
Remark~\ref{rmk:p=0}), we have 
        \begin{equation*}
\max\limits_{t\in[0,T]} \int_{\Tt^d}^{}m\,\dx= 1,
\end{equation*}
which gives the analogue to \eqref{m_int_r} with \(r=1\) for any \(d\in\Nn\). 

In one-dimensional case, \(d=1\), the condition \(r>\alpha\) was used only  to guarantee that the value of \(\gamma\) in \eqref{gamma<0} satisfies  \(\gamma<1\). Arguing as in the previous proof adapted
to the present case, we have
\begin{equation*}
\gamma=(1-\theta)\frac{q}{q+\alpha},
\end{equation*}
which satisfies the condition   \(\gamma<1\) for  all $q>1$. 

The proof of the case $d>2$ is also similar to the proof of  Proposition \ref{mIsLinfnity-q}. However, the proof is slightly different. Thus, although we omit the details, we outline next what needs to be changed. 

Regarding the $d>2$ case, we first observe that the \(d\)-dimensional version of \eqref{mIsLinfnity-q_eq1} holds
with $m^q$ in place of \(m^{-q}\) and the \(d\)-dimensional version of \eqref{Poincare} holds with \(m^{q+\alpha}\) in place of \(m^{-q+\alpha}\). In contrast with the $d=1$ case, 
we  use Sobolev's inequality instead of Morrey's embedding  theorem. So, by the  Sobolev inequality (see \cite[Theorem~6
in Section~5.6]{E6}), in place of \eqref{Morrey}, we have
\begin{equation*}
\left\Vert m^{\frac{q+\alpha}{2}}\right\Vert_{L^{2^*}(\Tt^d)}\leq {C}\left( \int_{\mathbb{T}^d}^{}m^{q+\alpha}\,\dx +\int_{\Tt^d}^{} \left|  D\left( m^{\frac{q+\alpha}{2}}\right) \right| ^2 \dx \right)^{\frac{1}{2}}. 
\end{equation*}
Set
 \begin{equation*}
\bar{\theta}=\frac{2q+\alpha d}{q(d(q+\alpha)-d+2)}.
\end{equation*}
It can be checked that $\bar{\theta}\in(0,1)$ and 
\begin{equation*}
\frac{1}{q}=\bar{\theta}+\frac{(1-\bar{\theta})(d-2)}{d(q+\alpha)}=\bar{\theta}+\frac{1-\bar{\theta}}{^{\frac{2^*}{2}(q+\alpha)}}.
\end{equation*}
Thus, by H\"older's  inequality, instead of \eqref{Holder1}, we get
\begin{equation*}
\left\Vert m\right\Vert_{L^{q}(\Tt^d)}\leq\left\Vert m\right\Vert^{\bar{\theta}} _{L^{1}(\Tt^d)} \left\Vert m\right\Vert^{1-\bar{\theta}} _{L^{\frac{2^*}{2}(q+\alpha)}(\Tt^d)}= \left\Vert m\right\Vert^{1-\bar{\theta}} _{L^{\frac{2^*}{2}(q+\alpha)}(\Tt^d)}.
\end{equation*}
Finally, here, we  use
\begin{equation*}
\bar{\gamma}=(1-\bar{\theta})\frac{q}{q+\alpha}
\end{equation*}
in place of \(\gamma\) in \eqref{gamma<0}.
Note that \(\bar0<\gamma<1\)  for  all $q>1$. The remaining  of the proof mimics that of Proposition \ref{mIsLinfnity-q}.

The  $d=2$ case is similar to the  $d>2$
case, using the fact that for any $1<a<2$,  the Sobolev inequality (see \cite[Section 5.6, Theorem 6]{E6}) yields\begin{equation*}
\left\Vert m^{\frac{q+\alpha}{2}}\right\Vert_{L^{a^*}(\Tt^2)}\leq C\left( \int_{\mathbb{T}^2}^{}m^{\frac{a(q+\alpha)}{2}}\,\dx +\int_{\Tt^2}^{} \left| D\left( m^{\frac{q+\alpha}{2}}\right) \right| ^a\,\dx \right)^{\frac{1}{a}}.
\end{equation*}
Then, by Young's inequality,  we obtain
\begin{equation*}\label{Sobolev_d=2}
\left\Vert m^{\frac{q+\alpha}{2}}\right\Vert_{L^{a^*}(\Tt^2)}\leq C\left( \int_{\mathbb{T}^2}^{}m^{q+\alpha}\,\dx +\int_{\Tt^2}^{} \left| D\left( m^{\frac{q+\alpha}{2}}\right) \right| ^2\,\dx\right)^{\frac{1}{a}}.\qedhere
\end{equation*}

\end{proof}

\begin{proof}[Proof of Theorem \ref{mIsLq-q+q}] 
The proof   follows from Propositions~\ref{mIsLinfnity-q}
and \ref{mIsLinfnity+q} and  Theorem \ref{mIsLq}.
\end{proof}

As we mentioned in the Introduction,
 we cannot expect, in general, bounds for \(m^{-1}\)
without a smallness condition on  \(V\),  as in Assumption~\ref{assumtionOf_V}. Next, we give an instance of  \eqref{planningPs} with an unbounded  potential that does not satisfy the conditions of Proposition \ref{mIsLinfnity-q}. In this particular case, we show that the density function, $m$,  has zero values. Hence,  the estimate \eqref{mIsLinfnity-q-st} does not hold without further conditions on \(V\).
\begin{example}\label{exam}Let 
        \begin{equation}\label{examp_m}
m(t,x)=1+\sin (2\pi x) \sin (2\pi t).
        \end{equation}
Notice that the function $m$ defined by \eqref{examp_m} is a probability density function and has two zeros, $(\frac{1}{4},\frac{3}{4})$ and $(\frac{3}{4},\frac{1}{4})$. Plugging     \eqref{examp_m} into the second equation in \eqref{planningPs} with $\alpha=1$  and solving it for $u$, we get
\begin{equation}\label{examp_u}
u(t,x)=-\frac{1}{2\pi}\cot (2 \pi t) \log(1 + \sin (2 \pi t) \sin(2 \pi  x)). 
\end{equation}
Set 
\begin{equation}\label{examp_V}
V(t,x)=m(x,t)+u_t(x,t)-\frac{u_x^2(x,t)}{2}.
\end{equation}
\end{example}
The functions, $m$ and $u$, given by \eqref{examp_m} and \eqref{examp_u}, respectively,  solve the mean-field planning problem 
\begin{equation*}
\begin{cases} 
-u_t+\tfrac{u_x^2}{2}+V(t,x)=m& \\ m_t-(u_x m)_x=0 & \\
m(0,x)=m(T,x)=1,&
\end{cases}
\end{equation*}
with $V$ given by \eqref{examp_V}. Direct computations  show that $\Delta V$ is unbounded, which means the function $V$ does not satisfy the conditions of Proposition \ref{mIsLinfnity-q}.

\bibliographystyle{plain}

\def\cprime{$'$}

\end{document}